\documentclass{amsart}

\usepackage[all]{xy}
\usepackage[normalem]{ulem}
\usepackage{amsmath,amssymb}
\usepackage{amsthm}
\usepackage{enumerate}
\usepackage{url}
\usepackage{graphicx}
\usepackage{esvect}
\usepackage{mathrsfs}

\usepackage{tikz}

\usepackage{caption}
\usepackage{hyperref}

\makeatletter
\def\@settitle{%
  \begin{center}%
    \baselineskip 14\p@
    \normalfont\bfseries\large 
    \@title
  \end{center}%
}

\usepackage{etoolbox}

\makeatletter
\patchcmd{\@setauthors}{\MakeUppercase}{\relax}{}{}
\makeatother

\makeatletter
\def\section{\@startsection{section}{1}{\z@}%
  {-2.5ex \@plus -1ex \@minus -.2ex}%
  {1.5ex \@plus .2ex}%
  {\normalfont\bfseries\centering}}
\makeatother

\theoremstyle{plain}
\newtheorem{theorem}{Theorem}[section]
\newtheorem{corollary}[theorem]{Corollary}
\newtheorem{lemma}[theorem]{Lemma}
\newtheorem{proposition}[theorem]{Proposition}

\theoremstyle{definition}
\newtheorem{definition}[theorem]{Definition}
\newtheorem{example}[theorem]{Example}
\newtheorem{remark}[theorem]{Remark}

\newtheorem*{thm*}{Theorem}

\newcommand{\ra}{\longrightarrow}
\def\co{\colon\thinspace}

\newcommand{\id}{\mathrm{id}}

\newcommand{\Id}{\mathrm{Id}}
\newcommand{\Lk}{\mathrm{Lk}}
\newcommand{\St}{\mathrm{St}}
\newcommand{\lk}{\mathrm{lk}}
\newcommand{\st}{\mathrm{st}}
\newcommand{\sd}{\mathrm{sd}}
\newcommand{\T}{\mathrm{T}}
\newcommand{\M}{\mathrm{M}}
\newcommand{\trh}{\mathrm{trh}}

\newcommand{\vep}{\varepsilon}
\newcommand{\N}{\mathbb{N}}
\newcommand{\R}{\mathbb{R}}

\newcommand{\Top}{\mathbf{Top}}

\newcommand{\hyphenmod}{{\operatorname{-Mod}}}
\newcommand{\Pos}{\mathbf{Pos}}
\newcommand{\Ho}{\mathrm{Ho}}

\newcommand{\bs}{\backslash}

\newcommand{\B}{\mathcal{B}}
\newcommand{\K}{\mathcal{K}}
\newcommand{\C}{\mathscr{C}}
\newcommand{\D}{\mathscr{D}}
\newcommand{\di}{d_{\mathrm{I}}}
\newcommand{\dhc}{d_{\mathrm{HC}}}

 \makeatletter
    
    \@addtoreset{equation}{section}
  \makeatother

\subjclass[2020]{}

\author{Vitalii Guzeev}
\address{Moscow Center for Continuous Mathematical Education, Bolshoy Vlasjevskiy 11, Moscow, Russia.}
\email{viviag@yandex.ru}

\author{Kohei Tanaka}
\address{Institute of Social Sciences, School of Humanities and Social Sciences, Academic Assembly, Shinshu University, 3-3-1 Asahi, Matsumoto, Nagano 390-8621, Japan.}
\email{tanaka@shinshu-u.ac.jp}

\title{Quillen-McCord theorem for persistence finite posets}

\begin{document}

\begin{abstract}
In this paper, we establish a persistence version of the Quillen-McCord theorem for persistence finite posets. 
Given a map $f \co P \to Q$ between persistence finite posets $P$ and $Q$ with weakly $\vep$-contractible homotopy fibers, we provide an upper bound for the homotopy commutative interleaving distance between $P$ and $Q$.
\end{abstract}

\maketitle

%%%%%%%%%%%%%%%%%%%%%%%%%%%%%%%%%%%
\section{Introduction}\label{sec:intro} % Sec 1
%%%%%%%%%%%%%%%%%%%%%%%%%%%%%%%%%%%

For a continuous map $f \co X \to Y$, studying its homotopy fiber $F_f$ is essential for understanding the relationship between $X$ and $Y$. 
In particular, if $F_f$ is weakly contractible (homotopically trivial), then it follows that $f$ is a weak homotopy equivalence. 

From a combinatorial viewpoint, the version of this statement applied to (finite) posets is known as the Quillen–McCord theorem (or Quillen fiber lemma) \cite{McC66, Qui72}.
When a finite poset is regarded as a finite $T_0$-space whose open sets are upper sets, the Quillen–McCord theorem can be described as follows since the McCord weak homotopy equivalence $\B(P) \to P$ for a finite poset $P$ and the classifying space $\B(P)$ \cite{McC66}.

\begin{theorem}[Quillen-McCord's theorem \cite{McC66, Qui72}]\label{th:QM}
Let $f \co P \to Q$ be an order preserving map between finite posets.
Assume that one of the following conditions holds:
\begin{enumerate}
\item $f^{-1}(L_v)$ is weakly contractible for each $v \in Q$, where $L_v = \{x \in Q \mid x \leq v\}$.
\item $f^{-1}(U_v)$ is weakly contractible for each $v \in Q$, where $U_v = \{x \in Q \mid x \geq v\}$.
\end{enumerate}
Then $f$ is a weak homotopy equivalence.
\end{theorem}

The following result is also known as a homological version of the above theorem.

\begin{theorem}[Homological Quillen-McCord's theorem \cite{Bar11a}]\label{th:HQM}
Let $R$ be a PID and let $H_*$ denote the homology with coefficients in $R$. 
For an order preserving map $f \co P \to Q$ between finite posets $P$ and $Q$, we assume that one of the following conditions holds:
\begin{enumerate}
\item $f^{-1}(L_v)$ is acyclic, i.e. $H_*(f^{-1}(L_v))$ is trivial for each $v \in Q$.
\item $f^{-1}(U_v)$ is acyclic, i.e. $H_*(f^{-1}(U_v))$ is trivial for each $v \in Q$.
\end{enumerate}
Then $f$ induces an isomorphism $H_*(P) \cong H_*(Q)$.
\end{theorem}

On the other hand, in the framework of TDA (Topological Data Analysis), one considers a sequence of spaces indexed by a numeric parameter and computes its persistence homologies.

One established technique of building a sequence of spaces is to start with a collection of data equipped with numeric indexes and consider spaces built from the said collection filtered by different threshold values of index (see \cite{SG07, KCM+21} for examples). 
In this technique, resulting object is naturally a filtration. The other known approach involves indexing by time (\cite{SSG+18}).

In either case, researchers face two problems.
	
\begin{enumerate}
		\item How to measure proximity of results of different experiments and when to consider them essentially equivalent.
		\item How to actually compute the results: computation of invariants like persistent homology heavily depends on the data size, for coefficients in a field specifically on dimension of spaces, and may have unacceptable costs.
	\end{enumerate}
    
    It seems possible that a persistence version of a Quillen-McCord theorem is useful for resolution of both problems.
    Specifically, given a suitable persistence version of a theorem one could design an algorithm to reduce dimension of persistence homology at the cost of sharpness of the result (\cite{Ayz, Guz}).

One of the tools used to compare results, namely persistent homology, is the interleaving distance denoted by $\di(M,N)$ for persistence modules $M,N$.
As a persistence version of the homological Quillen-McCord theorem (Theorem \ref{th:HQM}), the previous work \cite{Guz} provides an estimate on the interleaving distance of the associated persistent homologies $H_*(P)$ and $H_*(Q)$ for a map $f \co P \to Q$ with  $\vep$-acyclic homotopy fibers.

\begin{theorem}[Persistence homological Quillen-McCord's theorem \cite{Guz}]\label{th:Guz}
For a map $f \co P \to Q$ between suitable persistence finite posets $P$ and $Q$, we assume that one of the following conditions holds:
\begin{enumerate}
\item $f^{-1}(L_v)$ is $\vep$-acyclic for each persistence point $v$ in $Q$.
\item $f^{-1}(U_v)$ is $\vep$-acyclic for each persistence point $v$ in $Q$.
\end{enumerate}
Then we have the following inequality for each $i$:
\[
\di(H_i(P), H_i(Q)) \leq 4 \vep |Q|,
\]
where $|Q|$ denotes the number of persistence points in $Q$.
\end{theorem}

The aim of this paper is to establish a persistence version of the Quillen-McCord theorem (Theorem \ref{th:QM}). 
The primary focus is on comparisons at the level of persistence spaces, prior to passing to homology.

For persistence spaces, several notions of interleaving distance up to homotopy are known \cite{BL23, LS23}.
A natural one is the homotopy commutative interleaving distance $\dhc$, defined using interleavings in the homotopy category of topological spaces.

\begin{theorem}[Persistence Quillen-McCord's theorem (Theorem \ref{th:main})]
For a map $f \co P \to Q$ between suitable persistence finite posets $P$ and $Q$, we assume that one of the following conditions holds:
\begin{enumerate}
\item $f^{-1}(L_v)$ is weakly $\vep$-contractible for each persistence point $v$ in $Q$.
\item $f^{-1}(U_v)$ is weakly $\vep$-contractible for each persistence point $v$ in $Q$.
\end{enumerate}
Then we have the following inequality:
\[
\dhc(P, Q) \leq 2 \vep |Q|.
\]
\end{theorem}

As a consequence of our main result, one can improve the upper bound for the interleaving distance of homology.

\begin{corollary}[Corollary \ref{cor:di_H}]
For a map $f \co P \to Q$ between suitable persistence finite posets $P$ and $Q$, we assume that one of the following conditions holds:
\begin{enumerate}
\item $f^{-1}(L_v)$ is weakly $\vep$-contractible for each persistence point $v$ in $Q$.
\item $f^{-1}(U_v)$ is weakly $\vep$-contractible for each persistence point $v$ in $Q$.
\end{enumerate}
Then we have the following inequality for each $i$:
\[
\di(H_i(P), H_i(Q)) \leq 2 \vep |Q|.
\]
\end{corollary}

The remainder of this paper is organized as follows:
In Section 2, we review persistence theory and the homotopy theory of simplicial complexes. 
Since our arguments mainly rely on the order complexes associated with finite spaces (posets), we also provide the necessary preliminaries.

In Section 3, we prove the main theorem (Theorem \ref{th:main}). 
The key ingredient in the proof is an estimate of how the distance $\dhc$ changes when a reducible point is removed, as in Proposition \ref{prop:ep-reducible}. 
Combining this estimate with the reduction method based on mapping cylinders developed in \cite{Bar11a, Guz}, we obtain the result.

%%%%%%%%%%%%%%%%%%%%%%%%%%%%%%%%%%%%%%%%%%%%%%%%%%%%%%%%%%%%%%%%%%%%
\section{Preliminaries on persistence objects}\label{sec:preliminaries}% Sec 2
%%%%%%%%%%%%%%%%%%%%%%%%%%%%%%%%%%%%%%%%%%%%%%%%%%%%%%%%%%%%%%%%%%%%

To treat indexed objects such as vector spaces and topological spaces in a unified manner, we introduce the notions of persistence objects and interleavings in a category.

%%%%%%%%%%%%%%%%%%%%%%%%%%%%%%%%%%%%%%%%%%%%%%%%%%%%%%%%%%%%%%%%%%%%
\subsection{Persistence objects}\label{sec:persistence_objects}% Sec2.1
%%%%%%%%%%%%%%%%%%%%%%%%%%%%%%%%%%%%%%%%%%%%%%%%%%%%%%%%%%%%%%%%%%%%

A poset $P$ can be regarded as a category whose object set is $P$ and there exists at most one morphism between two objects.
Let $\N$ denote the poset of natural numbers with the standard total order.

\begin{definition}
Let $\C$ be a category. A {\em persistence object} in $\C$ (indexed by $\N$) is a functor $X \co \N \to \C$.
In other words, a persistence object $X$ consists of a family of objects $\{X_i\}_{i \in \N}$ and morphisms $\varphi^{X}_{i,j} \co X_i \to X_j$ in $\C$ for $i \leq j$ satisfying the following condition:
\begin{enumerate}
\item $\varphi^{X}_{i,i}=\id_{X_i}$ for $i \in \N$.
\item $\varphi^{X}_{j,k} \circ \varphi^{X}_{i,j} = \varphi^{X}_{i,k}$ for $i \leq j \leq k$.
\end{enumerate}
\end{definition}

The morphisms $\varphi^{X}_{i,j}$ are called {\em structure maps} and are denoted simply by $\varphi_{i,j}$ when there is no risk of confusion.
For a persistence object $X$ in $\C$, the structure maps are determined by a family of maps $\{\varphi_{i,i+1}\}_{i \in \N}$ because any $\varphi_{i,j}$ can be represented as
\[
\varphi_{i,j}= \varphi_{j-1,j} \circ \cdots \circ \varphi_{i+1,i+2} \circ \varphi_{i,i+1}.
\]

\begin{remark}
When dealing with continuously varying objects, it is common to regard $\R$, rather than $\N$, as a totally ordered set and to consider functors $\R \to \C$ as persistence objects in $\C$. 
However, since we mainly focus on discrete finite spaces (posets), we take $\N$ as the indexing category.
\end{remark}

\begin{example}
We give several typical examples of persistence objects.
\begin{enumerate}
\item Let $R$ be a PID. A functor $\N \to R\hyphenmod$ to the category $R\hyphenmod$ of $R$-modules is called a {\em persistence module} over $R$.
\item A functor $\N \to \Top$ to the category $\Top$ of topological spaces is called a {\em persistence space}.
\item A functor $\N \to \Pos$ to the category $\Pos$ of posets is called a {\em persistence poset}.
\end{enumerate}
\end{example}

A persistence object that stabilizes after some point is said to be of finite type.

\begin{definition}
A persistence object $X$ in $\C$ is said to be of {\em finite type} if there exists $i \in \N$ such that the structure map $\varphi_{j,k} \co X_{j} \to X_{k}$ is an isomorphism for all $i \leq j \leq k$.
\end{definition}

We introduce morphisms between persistence objects that take shifts into account.

\begin{definition}
For $\vep \in \N$ and two persistence objects $X$ and $Y$ in $\C$, an {\em $\vep$-map (or morphism)} $f \co X \to Y$ is a family of morphisms $\{f_i \co X_i \to Y_{i+\vep}\}_{i \in \N}$ in $\C$ making the following diagram commute for any $i \leq j$:
\[
\xymatrix{
X_{i} \ar[r]^{f_i} \ar[d]_{\varphi^{X}_{i,j}}& Y_{i+\vep} \ar[d]^{\varphi^{Y}_{i+\vep,j+\vep}}\\
X_{j} \ar[r]_{f_j} & Y_{j+\vep}.
 \ar@{}[ul]|{ \circlearrowleft }
}
\]
In particular, a $0$-map agrees with a natural transformation between the functors $X$ and $Y$, and we simply call it a {\em map (or morphism)} of persistence objects.
\end{definition}

A particularly important self $\vep$-map is the shift map defined using the structure maps.
Let $X$ be a persistence object in $\C$. The $\vep$-map $\Id^\vep \co X \to X$ is defined by $\Id^{\vep}_{i}=\varphi_{i,i+\vep}$.

\begin{definition}
Persistence objects $X$ and $Y$ in $\C$ are {\em $\vep$-interleaved} if there exist $\vep$-maps
$f \co X \to Y$ and $g \co Y \to X$ such that $g \circ f = \Id^{2\vep}$ and $f \circ g = \Id^{2\vep}$.
That is, for each $i \in \N$, the composition $g_{i+\vep} \circ f_{i} \co X_i \to X_{i+2 \vep}$ coincides with the structure map $\varphi^X_{i,i+2\vep}$, and the composition $f_{i+\vep} \circ g_{i} \co Y_i \to Y_{i+2 \vep}$ coincides with the structure map $\varphi^Y_{i,i+2\vep}$.
\end{definition}

Note that $0$-interleaved persistence objects are isomorphic as functors. 

For two persistence objects $X$ and $Y$ in a category $\C$, the {\em interleaving distance} $\di(X,Y)$ is defined as 
\[
d_{\mathrm{I}}(X,Y) = \inf\{ \vep \in \N \cup \{0\} \mid \textrm{$X$ and $Y$ are $\vep$-interleaved} \}.
\]
This is an extended pseudometric on persistence objects.

\begin{lemma}\label{lem:functor}
Let $F \co \C \to \D$ be a functor. 
If two persistence objects $X$ and $Y$ in $\C$ are $\vep$-interleaved, then $F(X)$ and $F(Y)$ are $\vep$-interleaved in $\D$. Moreover, we have $\di(F(X),F(Y)) \leq \di(X,Y)$.
\end{lemma}
\begin{proof}
This immediately follows from the fact that $F$ preserves compositions and identities.
\end{proof}

%%%%%%%%%%%%%%%%%%%%%%%%%%%%%%%%%%%%%%%%%%%%%%%%%%%%%%%%%%%%%%%%%%%%
\subsection{Persistence spaces}\label{sec:persistence_spaces}% Sec 2.2
%%%%%%%%%%%%%%%%%%%%%%%%%%%%%%%%%%%%%%%%%%%%%%%%%%%%%%%%%%%%%%%%%%%%

For persistence spaces, the notion of interleaving in Section \ref{sec:persistence_objects} can be extended to the up-to-homotopy version.
In principle, persistence spaces $\N \to \Top$ themselves could be extended to lax functors. 

However, since constructions such as mapping cylinders become technically complicated in that framework, we restrict our attention here to the up-to-homotopy version only for maps.

\begin{definition}
Let $X$ and $Y$ be persistence spaces. A {\em lax $\vep$-map} $f \co X \to Y$ is a family of maps $\{f_i \co X_i \to Y_{i+\vep}\}_{i \in \N}$ making the following diagram commute up to homotopy for any $i \leq j$:
\[
\xymatrix{
X_i \ar[r]^{f_i} \ar[d]_{\varphi^{X}_{i,j}}& Y_{i+\vep} \ar[d]^{\varphi^{Y}_{i+\vep,j+\vep}}\\
X_{j} \ar[r]_{f_j} & Y_{j+\vep}.
 \ar@{}[ul]|{ \simeq }
}
\]

\end{definition}

\begin{remark}
A lax natural transformation $X \Rightarrow Y$ (see \cite{Str72, JY21}) between persistence spaces $X$ and $Y$ yields a lax $0$-map.
Conversely, if we have a lax $0$-map $f \co X \to Y$, then a family of homotopies 
\[
\{ h_i \co \varphi^{Y}_{i,i+1} \circ f_i \Rightarrow f_{i+1} \circ \varphi^{X}_{i,i+1}\}_{i \in \N}
\]
yields a lax natural transformation $\alpha \co X \Rightarrow Y$. That is, $\alpha_i=f_i$ for $i \in \N$ and $\alpha_{i \leq j}$ is the homotopy given by the following composition of homotopies:
\[
\xymatrix{
X_i \ar[r]^{\varphi^{X}_{i,i+1}} \ar[d]_{f_i}& X_{i+1} \ar[r]^{\varphi^{X}_{i+1,i+2}}  \ar[d]^{f_{i+1}} & X_{i+2} \ar[r] \ar[d]^{f_{i+2}} & \cdots \ar[r]^{\varphi^{X}_{j-1,j}}   & X_j \ar[d]^{f_{j}}\\
Y_{i} \ar[r]_{\varphi^{Y}_{i,i+1}} & Y_{i+1} \ar@{}[ul]|{ \stackrel{h_{i}}{\Rightarrow} } \ar[r]_{\varphi^{Y}_{i+1,i+2}}  & Y_{i+2} \ar@{}[ul]|{ \stackrel{h_{i+1}}{\Rightarrow} }  \ar[r] & \cdots \ar[r]_{\varphi^{Y}_{j-1,j}}  & Y_j. \ar@{}[ul]|{ \stackrel{h_{j-1}}{\Rightarrow} }
}
\]
Indeed, $\alpha$ satisfies the coherence axiom with respect to identities and composition of lax natural transformations.

We simply call a lax $0$-map a {\em lax map}.
\end{remark}

There are several notions of interleavings up to homotopy for persistence spaces, or more generally persistence objects in a model category (see \cite{BL23, LS23}).
One simple idea is to consider interleavings in the homotopy category of topological spaces.
Here, the homotopy category $\Ho(\Top)$ is the localization of $\Top$ with respect to weak homotopy equivalences.

\begin{definition}
Persistence spaces $X$ and $Y$ are {\em homotopy commutative $\vep$-interleaved} if they are $\vep$-interleaved in the homotopy category $\Ho(\Top)$.
\end{definition}

More precisely, if $\Gamma \co \Top \to \Ho(\Top)$ is the localization, which is the identity on objects, then two persistence spaces $X$ and $Y$ are homotopy commutative $\vep$-interleaved if and only if $\Gamma(X)$ and $\Gamma(Y)$ are $\vep$-interleaved. 
We then define the {\em homotopy interleaving distance} $\dhc$ by
\[
\dhc(X,Y):= \di(\Gamma(X),\Gamma(Y)).
\]

\begin{remark}\label{rem:CW}
In the case that  $X$ and $Y$ are persistence CW complexes, the homotopy commutative interleavings between $X$ and $Y$ can be translated into lax maps by the Whitehead theorem. 
Namely, $X$ and $Y$ are homotopy commutative $\vep$-interleaved if and only if 
there exist lax $\vep$-maps
$f \co X \to Y$ and $g \co Y \to X$ such that $g \circ f \simeq \Id^X_{2 \vep}$ and $f \circ g \simeq \Id_{2 \vep}^{Y}$.

That is, for each $i \in \N$, the composition $g_{i+\vep} \circ f_{i} \co X_i \to X_{i+2 \vep}$ is homotopic to the structure map $\varphi^X_{i,i+2\vep}$ in $X$, and the composition $f_{i+\vep} \circ g_{i} \co Y_i \to Y_{i+2 \vep}$ is homotopic to the structure map $\varphi^Y_{i,i+2\vep}$ in $Y$.
\end{remark}

The following facts follow immediately from the definition of homotopy commutative interleavings.

\begin{proposition}\label{prop:interleaving}
We have the following properties for homotopy commutative interleavings:
\begin{enumerate}
\item $\vep$-interleaved persistence spaces are homotopy commutative $\vep$-interleaved.
\item If $X,Y$ are homotopy commutative $\vep$-interleaved and $Y, Z$ are homotopy commutative $\delta$-interleaved, then $X, Z$ are homotopy commutative $(\vep+\delta)$-interleaved.
\item Two persistence CW complexes $X$ and $Y$ are homotopy commutative $0$-interleaved if and only if we have a family of homotopy equivalences $\{f_i \co X_i \to Y_i\}_{i \in \N}$ compatible with structure maps up to homotopy.
\end{enumerate}
\end{proposition}

A persistence space $X$ is called {\em empty} if $X_i = \emptyset$ for all $i \in \N$.

\begin{definition}
The {\em threshold} $\trh(X) \in \N$ of a persistence space $X$ is defined as
\[
\trh(X)= \min \{ i \in \N \mid X_i \neq \emptyset\}.
\]
When $X$ is empty, we set $\trh(X)=\infty$.
\end{definition}

Let $*^{i}$ denote a non-empty persistence space with the threshold $i \in \N$ whose non-empty terms are all one-point spaces:
\[
\emptyset \ra \emptyset \ra \cdots \ra \emptyset \ra * \ra * \ra * \ra \cdots
\]

\begin{definition}\label{def:contractible}
A persistence space $X$ is {\em weakly $\vep$-contractible} if $X$ and $*^i$ are homotopy commutative $\vep$-interleaved for $i=\trh(X)$.
Moreover, $X$ is {\em $\vep$-acyclic} if $H_j(X)$ and $H_j(*^i)$ are $\vep$-interleaved modules for $i=\trh(X)$ and any $j$.
\end{definition}

The notion of $\vep$-acyclicity for persistence spaces was originally introduced in \cite{Guz}; however, in this paper, we introduce it with the threshold in order to ensure a more rigorous set-theoretic treatment. 

As we will see in Proposition \ref{prop:homology}, every weakly $\vep$-contractible space is $\vep$-acyclic.

\begin{definition}
A {\em persistence subspace} $Y$ of a persistence space $X$ is a persistence space satisfying the following conditions:
\begin{enumerate}
\item $Y_i$ is a subspace of $X_i$ for each $i \in \N$.
\item The structure map $\varphi^Y_{i,j} \co Y_i \to Y_j$ satisfies $\varphi^Y_{i,j}(y)=\varphi^{X}_{i,j}(y)$ for any $y \in Y_i$ and $i \leq j$.
\end{enumerate}
\end{definition}

The family of inclusions gives a map $Y \hookrightarrow X$ of persistence spaces.

\begin{definition}
Let $X$ be a persistence space. A {\em persistence point} of $X$ is a non-empty persistence subspace $x$ of $X$ such that the cardinality $|x_i| \leq 1$ for each $i \in \N$, and $\varphi_{i,j}^{-1}(x_j)=\{x_i\}$ for $i \geq \trh(x)$.
In this case, we write $x \in X$.
\end{definition}

\begin{remark}
For a persistence point $x$ in $X$, we define a persistence subspace $X \bs x$ of $X$ by $(X \bs x)_i=X_i \bs x_i$ for $i \in \N$. 
Note that the restriction of structure map $X_{i} \to X_{j}$ onto $X_i \bs x_i$ maps into $X_j \bs x_j$.
\end{remark}

A persistence space $X$ is called a {\em filtration} if every structure map $\varphi^X_{i,j}$ is injective.

\begin{example}\label{ex:point}
Let $X$ be a filtration. A point $x_i \in X_i$ determines a persistence point $x$ in $X$ by 
\[
x_j=
\begin{cases}
\varphi_{i,j}(x_i) & \textrm{if $i \leq j$,} \\
 \varphi_{j,i}^{-1}(x_i) &\textrm{if $ j<i$.}
\end{cases}
\]
Conversely, any persistence point $x$ in a filtration $X$ is uniquely determined by a point $x_i \in X_i$ for some $i \in \N$.
\end{example}

\begin{remark}
A filtration $X$ can be decomposed into persistence points; that is 
\[
X = \coprod_{x \in X}x
\]
as a persistence set.
This does not hold in general for persistence spaces (sets) that are not filtrations. 
This fact is essential to the reduction method for persistence points used in the proof of Theorem \ref{th:main}.
\end{remark}

%%%%%%%%%%%%%%%%%%%%%%%%%%%%%%%%%%%%%%%%%%%%%%%%%%%%%%%%%%%%%%%%%%%%
\subsection{Simplicial complexes}% Sec 2.3
%%%%%%%%%%%%%%%%%%%%%%%%%%%%%%%%%%%%%%%%%%%%%%%%%%%%%%%%%%%%%%%%%%%%

In developing a combinatorial homotopy theory, simplicial complexes are convenient to work with. 
For basic notions concerning (abstract) simplicial complexes, we refer the reader to \cite[Section 3]{Spa66}.

\begin{definition}
A {\em simplicial complex} $K$ consists of the set of vertices $V(K)$ and the set of simplices $\Sigma(K)$ which is a subset of the power set $2^{V(K)}$ satisfying the following conditions:
\begin{enumerate}
\item $\{v\} \in \Sigma(K)$ for any $v \in V(K)$,
\item $\sigma \subset \tau$ and $\tau \in \Sigma(K)$ imply $\sigma \in \Sigma(K)$.
\end{enumerate}
\end{definition}
We deal only with finite simplicial complexes which have finitely many vertices in this paper.

\begin{definition}
The {\em geometric realization} $||K||$ of a finite simplicial complex $K$ is the topological space constructed as follows:
Let $V(K)=\{v_0,\cdots v_n\}$ and we identify $v_i$ with the vector $(0,\cdots, 0,1,0, \cdots, 0) \in \mathbb{R}^{n+1}$ whose $(i+1)$-component is $1$ and all the others are $0$.
\[
||K|| = \left\{\sum_{i=0}^{n} \lambda_iv_i \mid \lambda_i \geq 0, \sum_{i=0}^{n} \lambda_i=1, \{v_i \mid \lambda_i>0\} \in \Sigma(K) \right\}.
\]
\end{definition}

\begin{definition}
Let $K$ and $L$ be simplicial complexes. The {\em join} $K \star L$ is a simplicial complex with $V(K \star L)=V(K) \coprod V(L)$ and 
\[
\Sigma(K \star L)=\{\sigma \cup \tau \mid \sigma \in \Sigma(K), \tau \in \Sigma(L)\}.
\]
\end{definition}

The geometric realization $||K \star L||$ of the simplicial join $K \star L$ is naturally homeomorphic to the topological join $||K|| \star ||L||$.

\begin{definition}
Let $K$ be a simplicial complex and $v \in V(K)$ be a vertex.
\begin{enumerate}
\item The {\em star} $\st_K(v)$ is a subcomplex of $K$ that consists of
\[
\Sigma(\st_K(v))=\{ \sigma \in \Sigma(K) \mid \sigma \cup \{v\} \in \Sigma(K)\}.
\] 
\item The {\em link} $\lk_K(v)$ is a  subcomplex of $K$ (and $\st(v)$) that consists of
\[
\Sigma(\lk_K(v))=\{ \sigma \in \Sigma(K) \mid v \not \in \sigma, \sigma \cup \{v\} \in \Sigma(K)\}.
\] 
\end{enumerate}
We denote them simply by $\st(v)$ and $\lk(v)$ when there is no risk of confusion.
Furthermore, the geometric realizations are denoted by $\St(v)=||\st(v)||$ and $\Lk(v)=||\lk(v)||$.
\end{definition}

For a vertex $v$ of a simplicial complex $K$, we have $\st(v)=\lk(v) \star v$.

\begin{definition}
Let $P$ be a finite poset. The {\em order complex} $\K(P)$ is the simplicial complex with $V(\K(P))=P$ such that $\Sigma(\K(P))$ consists of totally ordered subsets in $P$.
The geometric realization $||\K(P)||$ is denoted by $\B(P)$ and is called the {\em classifying space} of $P$.
\end{definition}

A {\em simplicial map} $\varphi \co K \to L$ between simplicial complexes $K$ and $L$ is a map on vertices $\varphi \co V(K) \to V(L)$ which sends each simplex of $K$ to a simplex $L$.

An order preserving map $f \co P \to Q$ induces a simplicial map $f_{\K} \co \K(P) \to \K(Q)$ defined by $f_{\K}(p)=f(p)$.
Furthermore, a simplicial map $\varphi \co K \to L$ induces a continuous map $||\varphi|| \co ||K|| \to ||L||$ defined by linear extension:
\[
||\varphi|| \left(\sum_{v \in V(K)}\lambda_v v \right) = \sum_{v \in V(K)} \lambda_v \varphi(v).
\]

\begin{definition}
Two simplicial maps $\varphi, \psi \co K \to L$ are called {\em contiguous} if $\varphi(\sigma) \cup \psi(\sigma)$ is a simplex of $L$ for each simplex $\sigma$ in $K$.
In this case, we write $\varphi \sim_{\mathrm{c}} \psi$.
\end{definition}

The contiguity relation on simplicial maps is reflexive and symmetric, but not transitive in general.
An equivalence relation on simplicial maps generated from $\sim_{\mathrm{c}}$ is denoted by $\sim$.
That is, $\varphi \sim \psi \co K \to L$ if and only if there exists a finite number of simplicial maps $\varphi=\omega_1,\omega_2, \cdots, \omega_n=\psi$ from $K$ to $L$ such that $\omega_i \sim_{\mathrm{c}} \omega_{i+1}$ for each $i$.
In this case, we say that $\varphi$ and $\psi$ are {\em in the same contiguity class}.

\begin{lemma}[Lemma 3.5.2 in \cite{Spa66}]
If two simplicial maps $\varphi, \psi \co K \to L$ are in the same contiguity class, then the induced maps $||\varphi||, ||\psi|| \co ||K|| \to ||L||$ are homotopic.
\end{lemma}

Let $\varphi \co K \to L$ be a simplicial map and let $v$ be a vertex of $K$.
The restriction of $\varphi$ induces 
\[
\varphi_{\st}=\varphi|_{\st_K(v)} \co \st_K(v) \ra \st_L(\varphi(v)).
\]
In addition,  if $\varphi$ satisfies $\varphi^{-1}(\varphi(v))=\{v\}$, the restriction induces
\[
\varphi_{\lk}=\varphi|_{\lk_K(v)} \co \lk_K(v) \ra \lk_L(\varphi(v)).
\]
Let $\varphi_{\St}=||\varphi_{\st}||$ and $\varphi_{\Lk}=||\varphi_{\lk}||$ denote the induced map on the geometric realizations.

We consider the following commutative diagram of spaces for a simplicial map $\varphi \co K \to L$ satisfying $\varphi^{-1}(\varphi(v))=\{v\}$:
\[
\xymatrix{
\Lk_K(v) \ar[r]^{j} \ar[d]_{\varphi_{\Lk}} & \St_K(v) \ar[d]^{\varphi_{\St}} \\
\Lk_{L}(\varphi(v)) \ar[r]^{k} & \St_L(\varphi(v))
}
\]
where the horizontal maps $j,k$ are the inclusions.

\begin{lemma}\label{lem:contiguity}
Let $\varphi \co K \to L$ be a simplicial map satisfying $\varphi^{-1}(\varphi(v))=\{v\}$.
If the simplicial map $\varphi_{\lk}$ and the constant map onto a vertex $u \in \lk_{L}(\varphi(v))$ are in the same contiguity class, then there exists a continuous map $\tilde{\varphi} \co  \St_K(v) \to \Lk_{L}(\varphi(v))$ satisfying the following conditions:
\begin{enumerate}
\item $\tilde{\varphi} \circ j=\varphi_{\Lk}$, i.e., the left upper triangle in the following diagram is commutative.
\item $k \circ \tilde{\varphi} \simeq \varphi_{\St}$ relative to $\Lk_K(v)$, i.e., the lower right triangle in the following diagram is commutative up to homotopy relative to $\Lk_K(v)$.
\end{enumerate}
\[
\xymatrix{
\Lk_K(v) \ar[r]^{j} \ar[d]_{\varphi_{\Lk}} & \St_K(v) \ar[d]^{\varphi_{\St}} \ar@{..>}[ld]_{\tilde{\varphi}} \\
\Lk_{L}(\varphi(v)) \ar[r]^{k} & \St_L(\varphi(v))
}
\]
\end{lemma}
\begin{proof}
We have a finite number of simplicial maps $\varphi_{\lk}=\psi_0,\psi_1,\cdots,\psi_n=c_u$ such that $\psi_i \sim_{\mathrm{c}} \psi_{i+1} \co \lk_K(v) \to \lk_{L}(\varphi(v))$ for each $i$, where $c_u$ denotes the constant map onto $u$.

Note that the star $\St_K(v)=\Lk_K(v) \star v$ is the cone of $\Lk_K(v)$, and an element of $\St_K(v)$ can be expressed as $[x,t]$ for $x \in \Lk_K(v)$ and $t \in [0,1]$, where $[x,1]=[y,1]$ for any $x,y \in \Lk_K(v)$.

A continuous map $\tilde{\varphi} \co \St_K(v) \to \Lk_{L}(\varphi(v))$ is defined by
\[
\tilde{\varphi}[x,t] = (1-t')||\psi_{i-1}||(x)+t'||\psi_i||(x)
\]
for $t'=nt-(i-1)$ when $\frac{i-1}{n} \leq t \leq \frac{i}{n}$. This map is well-defined because of the contiguity between $\psi_{i-1}$ and $\psi_i$.
Moreover, $\tilde{\varphi}$ satisfies the first condition because 
\[
\tilde{\varphi} \circ j (x)= \tilde{\varphi}[x,0]=||\psi_0||(x)=\varphi_{\Lk}(x).
\]
We define continuous maps $h_0,h_1,\cdots,h_n \co \St_K(v) \to \St_L(\varphi(v))$ by
\[
h_{i}[x,t] = \begin{cases} k \circ \tilde{\varphi}[x,t] & \textrm{if $0 \leq t \leq \frac{i}{n}$}, \\
[\psi_{i}(x),t_{i}] &  \textrm{if $\frac{i}{n} \leq t \leq 1$}, 
\end{cases}
\]
where $t_{i}=\frac{nt-i}{n-i}$. The maps $h_i$ and $h_{i+1}$ are homotopic relative to $\Lk_K(v)$ because $h_i[x,t]$ and $h_{i+1}[x,t]$ are in a common simplex of $\St_L(\varphi(v))$ by the contiguous maps $\psi_i$ and $\psi_{i+1}$.
Hence, $\varphi_{\St}=h_0$ and $k \circ \tilde{\varphi}=h_n$ are homotopic relative to $\Lk_K(v)$.
\end{proof}

\begin{corollary}\label{cor:simplicial_app}
Consider the commutative diagram just before Lemma \ref{lem:contiguity}. If $\varphi_{\Lk}$ is homotopic to the constant map onto a vertex $u \in \lk_{L}(\varphi(v))$, then there exists a continuous map $\tilde{\varphi} \co  \St_K(v) \to \Lk_{L}(\varphi(v))$ satisfying the same conditions as  Lemma \ref{lem:contiguity}.
\end{corollary}
\begin{proof}
We use the simplicial approximation theorem. 
Let $\sd(M)$ denote the barycentric subdivision of a simplicial complex $M$ that consists of barycenters of simplices in $M$.

By \cite[Theorem 3.5.6]{Spa66}, there exists a simplicial approximation 
\[
\omega \co \sd^{n}(\lk_K(v)) \to \lk_L(\varphi(v))
\]
of $\varphi_{\Lk}$ for sufficiently large $n$ such that $\omega$ and the constant map onto $u$ are in the same contiguity class. 
Note that every simplicial approximation to the constant map onto the vertex $u$ is constant onto $u$.

Because $\varphi_{\Lk}(\alpha)$ and $||\omega||(\alpha)$ are in a common simplex of $\Lk_L(\varphi(v))$ for any $\alpha \in ||\sd^{n}(\lk_K(v))||$, we can construct $\tilde{\varphi}$ in a similar way to Lemma  \ref{lem:contiguity} by spanning 
$\varphi_{\Lk}, ||\omega||, ||\psi_1||,\cdots, ||\psi_n||$, where $\omega=\psi_0, \psi_1,\cdots,\psi_n=c_u$ are simplicial maps $\sd^n(\lk_K(v)) \to \lk_L(\varphi(v))$ such that $\psi_i \sim_{\mathrm{c}} \psi_{i+1}$ for each $i$.
We can verify that $\tilde{\varphi}$ satisfies the desired conditions by a similar argument to that in the proof of Lemma \ref{lem:contiguity}.
\end{proof}

%%%%%%%%%%%%%%%%%%%%%%%%%%%%%%%%%%%%%%%%%%%%%%%%%%%%%%%%%%%%%%%%%%%%
\section{A persistence Quillen–McCord theorem for finite persistence posets}\label{sec:persistence_poset}% Sec 3
%%%%%%%%%%%%%%%%%%%%%%%%%%%%%%%%%%%%%%%%%%%%%%%%%%%%%%%%%%%%%%%%%%%%

%%%%%%%%%%%%%%%%%%%%%%%%%%%%%%%%%%%%%%%%%%%%%%%%%%%%%%%%%%%%%%%%%%%%
\subsection{Persistence finite posets (spaces)}% Sec 3.1
%%%%%%%%%%%%%%%%%%%%%%%%%%%%%%%%%%%%%%%%%%%%%%%%%%%%%%%%%%%%%%%%%%%%

A (finite) poset $P$ can be regarded as a topological space with the Alexandroff topology that consists of upper sets.
Here, an upper set (resp. a lower set) of $P$ is a subset closed upward (resp. downward) in $P$.

Conversely, a finite $\T_0$-space $X$ can be regarded as a poset with the specialization order $x \leq y$ defined by $U_x \supset U_y$, where $U_x$ and  $U_y$ are the minimal open neighborhoods of $x$ and $y$ respectively.
Furthermore, a map $f \co X \to Y$ between finite $\T_0$-spaces $X,Y$ is continuous if and only if $f$ is order-preserving.

From this perspective, we identify finite posets with finite $\T_0$-spaces (see \cite{Bar11b, McC66} for details).
Hence, persistence finite posets can be viewed as persistence spaces that we have seen in Section \ref{sec:persistence_spaces}.

\begin{proposition}\label{prop:B_interleaving}
Persistence finite posets $P$ and $Q$ are homotopy commutative $\vep$-interleaved if and only if the induced persistence spaces $\B(P)$ and $\B(Q)$ are homotopy commutative $\vep$-interleaved.
\end{proposition}
\begin{proof}
The result follows from the fact that $R$ and $\B(R)$ are isomorphic in $\Ho(\Top)$ for any finite poset $R$ by McCord's weak homotopy equivalence $\B(R) \to R$ \cite{McC66}.
This map gives natural isomorphisms $\B(P) \cong P$ and $\B(Q) \cong Q$ as functors $\N \to \Ho(\Top)$.
In other words, $\B(P)$ and $P$ ($\B(Q)$ and $Q$) are homotopy commutative $0$-interleaved.
Hence, $P$ and $Q$ are homotopy commutative $\vep$-interleaved if and only if  $\B(P)$ and $\B(Q)$ are homotopy commutative $\vep$-interleaved.
\end{proof}

\begin{definition}
Let $P$ be a poset. For a point $v \in P$, we use the following special lower sets and upper sets in this paper.
\begin{enumerate}
\item The principal lower set $L_v=\{u \in P \mid u \leq v\}$. Furthermore, $\hat{L}_v=L_v \bs \{v\}$.
\item The principal upper set $U_v=\{u \in P \mid u \geq v\}$. Furthermore, $\hat{U}_v=U_v \bs \{v\}$.
\end{enumerate}
\end{definition}

Note that for a persistence point $v$ of a persistence poset $P$, the persistence posets $L_v, \hat{L}_v, U_v, \hat{U}_v$ are well-defined as persistence subposets of $P$ with the same threshold as $v$.

\begin{definition}
Let $P$ be a persistence finite poset. A persistence point $v$ in $P$ is {\em $\vep$-reducible} if either $\hat{L}_v$ or  $\hat{U}_v$ is weakly $\vep$-contractible (Definition \ref{def:contractible}).
\end{definition}

Note that  $v$ is $\vep$-reducible in $P$ if and only if either the induced persistence space $\B(\hat{L}_v)$ or $\B(\hat{U}_v)$ is weakly $\vep$-contractible by Proposition \ref{prop:B_interleaving}.

\begin{proposition}\label{prop:ep-reducible}
Let $P$ be a persistence finite poset. 
If $v$ is an $\vep$-reducible persistence point in $P$, then the persistence posets $P$ and $P \bs v$ are homotopy commutative $2\vep$-interleaved.
\end{proposition}
\begin{proof}
By Proposition \ref{prop:B_interleaving}, it suffices to show that $\B(P)$ and $\B(P \bs v)$ are homotopy commutative $2\vep$-interleaved.

We have a $2\vep$-map $f \co \B(P \bs v) \to \B(P)$ defined as follows: for each $i \in \N$, the map $f_i$ is given by the composition of the structure map $\B(P \bs v)_i \to \B(P \bs v)_{i+2\vep}$ and the inclusion $j \co \B(P \bs v)_{i+2\vep} \hookrightarrow \B(P)_{i+2\vep}$.

On the other hand, a lax $2\vep$-map $g \co \B(P) \to \B(P \bs v)$ is constructed as follows:
In the case of $v_i=\emptyset$, we have 
\[
\B(P \bs v)_i = \B(P_i \bs v_i)=\B(P_i)=\B(P)_i.
\] 
Hence, we define $g_i \co \B(P)_i \to \B(P \bs v)_{i+2\vep}$ as the structure map of $\B(P)$.

We focus on the case of  $v_t \neq \emptyset$, i.e., it consists of a single point $v_t$.
Let us consider the following commutative diagram:
\[
\xymatrix@=2pc{
 \Lk(v)_t \ar[rr] \ar[dd]_{* \simeq} \ar[dr]& & \B(P \bs v)_t \ar'[d][dd]  \ar[dr] \\
& \St(v)_t \ar[rr] \ar[dd]  \ar@{..>}[dl]_{\tilde{\varphi}}& & \B(P)_t \ar[dd] \ar@{..>}[dl]_{g_t}  \\
 \Lk(v)_{t+2\vep}  \ar'[r][rr]  \ar[dr]_{k} & & \B(P \bs v)_{t+2\vep} \ar[dr]^{\ell} \\
& \St(v)_{t+2\vep} \ar[rr] & & \B(P)_{t+2\vep}.
}
\]
Here, the top and bottom squares are pushout diagrams given by inclusions, and the vertical maps are structure maps.
The assumption implies that the structure map 
\[
\varphi^{\Lk(v)}_{t,t+2\vep} \co \Lk(v)_t \ra \Lk(v)_{t+2\vep}
\]
of the link $\Lk(v)=\B(\hat{L}_v) \star \B(\hat{U}_v)$ is homotopic to the constant map onto a vertex.
Corollary \ref{cor:simplicial_app} gives an extension $\tilde{\varphi} \co \St(v)_t \to \Lk(v)_{t+2\vep}$ of $\varphi^{\Lk(v)}_{t,t+2\vep}$ with a homotopy $H$ between $k \circ \tilde{\varphi}$ and $\varphi^{\St(v)}_{t,t+2\vep}$ relative to $\Lk(v)_t$. 
Because the top square is the pushout diagram, we obtain an extension $g_t \co \B(P)_t \to \B(P \bs v)_{t+2\vep}$ of $\varphi^{\B(P \bs v)}_{t,t+2\vep}$ and $\tilde{\varphi}$.

The homotopy $H$ induces $\ell \circ g_t \simeq \varphi^{\B(P)}_{t,t+2\vep}$ and $g \co \B(P) \to \B(P \bs v)$ is a lax $2\vep$-map. 
By chasing the diagram, we have $f \circ g \simeq \Id^{\B(P)}_{4 \vep}$ and $g \circ f = \Id^{\B(P \bs v)}_{4 \vep}$. Thus, $\B(P)$ and $\B(P \bs v)$ are homotopy commutative $2\vep$-interleaved by Remark \ref{rem:CW}.
\end{proof}

%%%%%%%%%%%%%%%%%%%%%%%%%%%%%%%%%%%%%%%%%%%%%%%%%%%%%%%%%%%%%%%%%%%%
\subsection{Persistence Quillen-McCord theorem}% Sec 3.2
%%%%%%%%%%%%%%%%%%%%%%%%%%%%%%%%%%%%%%%%%%%%%%%%%%%%%%%%%%%%%%%%%%%%

\begin{definition}
Let $f \co P \to Q$ be an order-preserving map between posets $P,Q$. 
The {\em mapping cylinder} $\M_f$ of $f$ is a poset defined as follows: The underlying set of $\M_f $ is the coproduct $P \coprod Q$. 
The partial order $x \leq y$ on $\M_f$ is defined by 
\begin{itemize}
\item $x \leq_{P} y$ for $x,y \in P$,
\item  $x \leq_{Q} y$ for $x,y \in Q$,
\item $f(x) \leq_Q y$ for  $x \in P$ and $y \in Q$.
\end{itemize}
\end{definition}

For a map $f \co P \to Q$ between persistence posets, we also define the mapping cylinder $\M_f$ as a persistence poset by $(\M_f)_i=\M_{f_i}$.
The mapping cylinder $\M_f$ is equipped with maps of persistence posets $P \to \M_{f}$ and $Q \to \M_f$ given by component-wise inclusion.
Furthermore, we have a map $r \co M_{f} \to Q$ of persistence posets defined by 
\[
r_i(x)= \begin{cases} f_i(x) & \textrm{if $x \in P_i$}, \\
x & \textrm{if $x \in Q_i$}.
\end{cases}
\]

\begin{lemma}\label{lem:MfQ}
Let $f \co P \to Q$ be a map between persistence finite posets $P$ and $Q$.
The mapping cylinder $\M_f$ and $Q$ are homotopy commutative $0$-interleaved.
\end{lemma}
\begin{proof}
We can verify that the map $r \co \M_f \to Q$ and the inclusion $k \co Q \to \M_f$ exhibit $\M_f \simeq Q$. 
Indeed $r_i \circ k_i = \id_{Q_i}$ and $k_i \circ r_i \simeq \id_{(\M_f)_i}$ for each $i \in \N$. Hence,  $\M_f$ and $Q$ are isomorphic in $\Ho(\Top)$.
\end{proof}

\begin{definition}
For a persistence finite poset $P$, the {\em cardinality} $|P|$ is defined as
\[
|P| = \max \{ |P_i| \mid i \in \N\},
\]
if it exists.
\end{definition}

A finite type filtration $P$ of finite posets always has cardinality $|P|=|P_i|$ for sufficiently large $i \in \N$.

\begin{theorem}\label{th:main}
Let $f \co P \to Q$ be a map between persistence finite posets $P$ and $Q$, where $Q$ is a finite type filtration. 
We assume that one of the following conditions holds:
\begin{enumerate}
\item $f^{-1}(L_v)$ is weakly $\vep$-contractible for each persistence point $v$ in $Q$.
\item $f^{-1}(U_v)$ is weakly $\vep$-contractible for each persistence point $v$ in $Q$.
\end{enumerate}
Then $P$ and $Q$ are homotopy commutative $2\vep |Q|$-interleaved.
\end{theorem}
\begin{proof}
We prove only case (1), since (2) can be shown dually.
Let $Q_{\infty}$ denote $Q_i$ for a sufficiently large $i$ satisfying $Q_i \cong Q_j$ for any $j>i$.
Recall that a point $v \in Q_{\infty}$ can be regarded as a persistence point in $Q$ in Example \ref{ex:point}.

First, we choose a minimal point $v$ in $Q_{\infty}$. Note that each $v_s \subset Q_s$ consists of a minimal point in $Q_s$ or empty.
The lower set $\hat{L}_v$ in $\M_f$ agrees with $f^{-1}(L^{Q}_v)$. 
Hence, $v$ is $\vep$-reducible in $\M_f$.
Proposition \ref{prop:ep-reducible} implies that $\M_f$ and $\M_f \bs v$ are  homotopy commutative $2\vep$-interleaved.

Next, we choose a minimal point $u \in Q_{\infty} \bs v$. Again, the lower-set $\hat{L}_u$ in $\M_f \bs v$ agrees with $f^{-1}(L_u^Q)$ and $u$ is $\vep$-reducible in $\M_f \bs v$.
Hence, $\M_f \bs v$ and $\M_f \bs \{v,u\}$ are homotopy commutative $2\vep$-interleaved.

By repeating this process, Proposition \ref{prop:interleaving} implies that $\M_f$ and $P$ are homotopy commutative $2\vep|Q|$-interleaved.
The result follows from Lemma \ref{lem:MfQ} and Propositions \ref{prop:B_interleaving} and \ref{prop:interleaving}.
\end{proof}

For a PID $R$, the singular homology functor 
	\[
	H_*=H_*(-,R) \co \Top \ra R\hyphenmod
	\]
induces $\Ho(\Top) \to R\hyphenmod$ because homology sends weak homotopy equivalences to isomorphisms.
The following proposition follows from Lemma \ref{lem:functor}.

\begin{proposition}\label{prop:homology}
If two persistence spaces $X$ and $Y$ are homotopy commutative $\vep$-interleaved, then the persistent homologies $H_*(X,R)$ and $H_*(Y,R)$ are $\vep$-interleaved.
\end{proposition}

In the previous work \cite{Guz}, a homological version of our Theorem \ref{th:main} was established. 
It states that if the homotopy fibers of a map $P \to Q$ between finite type persistence finite posets are $\vep$-acyclic, then the persistent homologies $H_*(P)$ and $H_*(Q)$ are $4\vep |Q|$-interleaved.

Theorem \ref{th:main} and Proposition \ref{prop:homology} improve the corresponding result on homological interleavings in \cite{Guz}.

\begin{corollary}\label{cor:di_H}
Let $R$ be a PID and let $f \co P \to Q$ be a map between persistence finite posets $P$ and $Q$, where $Q$ is a finite type filtration. 
We assume that one of the following conditions holds:
\begin{enumerate}
\item $f^{-1}(L_v)$ is weakly $\vep$-contractible for each persistence point $v$ in $Q$.
\item $f^{-1}(U_v)$ is weakly $\vep$-contractible for each persistence point $v$ in $Q$.
\end{enumerate}
Then $H_*(P,R)$ and $H_*(Q,R)$ are $2\vep |Q|$-interleaved.
\end{corollary}
\begin{proof}
The result follows from Theorem \ref{th:main} and Lemma \ref{prop:homology}.
\end{proof}

%%%%%%%%%%%%%%%%%%%%%%%%%%%%%%%%%%%%%%%%%%%%%%%%%%%%%%%%%%%%%%%%%%%%
\subsection*{Concluding remarks and future directions}
%%%%%%%%%%%%%%%%%%%%%%%%%%%%%%%%%%%%%%%%%%%%%%%%%%%%%%%%%%%%%%%%%%%%

In this paper, we established a persistence version of the Quillen–McCord theorem for finite persistence posets.
Our main result (Theorem \ref{th:main}) implies that if there is a map $P \to Q$ from a persistence poset $P$ to a finite type filtration $Q$ with weakly $\vep$-contractible homotopy fibers, then we have
\[
d_{\mathrm{HC}}(P,Q) \leq 2\vep|Q|.
\]

Based on the present results, we would like to outline several directions for future investigation.

\begin{enumerate}
\item It is worthwhile to examine the sharpness of the estimates for $d_{\mathrm{HC}}$ obtained in this work. 
Although we have provided an upper bound for $d_{\mathrm{HC}}$, it remains to determine how tight this bound is. 
To this end, one should investigate concrete examples of persistence finite posets $P,Q$, and establish methods for computing $d_{\mathrm{HC}}$.
\item For persistence spaces $X,Y$, Proposition \ref{prop:homology} induces the following inequality:
\[
d_{\mathrm{I}}(H_*(X),H_*(Y)) \leq d_{\mathrm{HC}}(X,Y).
\]
The difference between $d_{\mathrm{I}}$ and $d_{\mathrm{HC}}$ is of interest as a way to explicitly capture the change in information after passing to homology. 

Since computations in general settings are expected to be difficult, it may be worthwhile to consider special cases—for instance, when $X$ and $Y$ are finite spaces (posets) satisfying the assumptions of Theorem \ref{th:main}—where this difference can be computed.
\item For a given finite poset (space) $P$, it is known that there exists, up to isomorphism, a unique minimal finite poset $P'$ that has the same homotopy type as $P$ \cite{Sto66}.
From the viewpoint of applications to TDA, it is of particular interest to formulate and construct a minimal model (in an appropriate sense) for a given persistence finite poset.
In fact, in our main theorem (Theorem \ref{th:main}), the upper bound depends on the size of $Q$; therefore, it may be possible to improve the estimate by replacing $Q$ with a smaller persistence finite poset $Q'$.

Furthermore, for a given persistence finite poset $Q$ and $\vep \geq 0$, it is of practical interest to construct a smaller model $P$ than $Q$ satisfying the conditions of Theorem \ref{th:main}.

\end{enumerate}

%%%%%%%%%%%%%%%%%%%%%%%%%%%%%%%%%%%%%% Bib

\end{document}